\documentclass[graybox]{svmult}

\usepackage{mathptmx}       
\usepackage{helvet}         
\usepackage{courier}        
\usepackage{type1cm}        
\usepackage{makeidx}         
\usepackage{graphicx}        
\usepackage{multicol}        
\usepackage[bottom]{footmisc}

\usepackage{amsfonts}
\usepackage{mathrsfs}    
\usepackage{amsmath}     

\makeindex             

\begin{document}

\title*{Spectral Measures on Locally Fields}
\author{Aihua Fan}
\institute{
Aihua Fan \at LAMFA, UMR 7352 CNRS, University of Picardie,
33 rue Saint Leu, 80039 Amiens, France, \email{ai-hua.fan@u-picardie.fr}}
\maketitle

\abstract*{In this paper, we propose to study spectral measures on local fields.
Some basic
results are presented,  including the stability of Bessel sequences under perturbation, the Landau theorem on Beurling density, the law of pure type
of spectral measures, the boundedness of the Radon-Nikodym derivative of
absolutely continuous $F$-spectral measures etc.}

\abstract{In this paper, we propose to study spectral measures on local fields.
Some basic
results are presented,  including the stability of Bessel sequences under perturbation, the Landau theorem on Beurling density, the law of pure type
of spectral measures, the boundedness of the Radon-Nikodym derivative of
absolutely continuous $F$-spectral measures etc.}

\section{Introduction}

Let $K$ be a local field of absolute value $|\cdot|$. The ring of integers in $K$
is denoted by $\mathfrak{D}$ and the Haar measure on $K$ is denoted by $\mathfrak{m}$ or $dx$.
We assume that the Haar measure is normalized so that $\mathfrak{m}(\mathfrak{D})=1$. The ring $\mathfrak{D}$ is the unique maximal compact subring of $K$ and it is the unit ball
$\{x\in K: |x|\le 1\}$. The ball $\{x\in K: |x|<1\}$, denoted $\mathfrak{P}$, is the maximal ideal in $\mathfrak{D}$ and it is principal and prime. The residue class field of $K$ is the field
$\mathfrak{D}/\mathfrak{P}$, which will be denoted by $\mathbf{k}$. Let $\mathfrak{p}$ be a fixed element of $\mathfrak{P}$ of maximal absolute value, called a prime element  of $K$. As an ideal in
$\mathfrak{D}$, $\mathfrak{P}= (\mathfrak{p})=\mathfrak{p}\mathfrak{D}$.  The residue class field
$\mathbf{k}$ is isomorphic to a finite field $\mathbb{F}_q$ where $q=p^c$ is a power of some
prime number $p\ge 2$ ($c\ge 1$ being an integer).
The dual group $\widehat{K}$ of $K$
is isomorphic to $K$. We fix a character $\chi \in \widehat{K}$ such that $\chi$ is equal to $1$ on $\mathfrak{D}$
but is non-constant on $\mathfrak{p}^{-1}\mathfrak{D}$. Then the map $y \mapsto \chi_y$ from
$K$ onto $\widehat{K}$ is an isomorphism, where $\chi_y(x) = \chi(yx)$.

For $d\ge 1$,
$K^d$ denotes the $d$-dimensional $K$-vector space. We endow $K^d$ with the norm
$$
|x|=\max_{1\le j\le d} |x_j|, \quad \mbox{\rm for} \ x = (x_1, \cdots, x_d)\in K^d.
 $$
 The Haar measure on $K^d$
is the product measure $dx_1\cdots dx_d$ which is also denoted by $\mathfrak{m}$, or
$\mathfrak{m}_d$ if it is necessary to point out the dimension.
For $x=(x_1, \cdots, x_d)\in K^d$ and $y=(y_1, \cdots, y_d) \in K^d$, we define
$$
   x\cdot y = x_1 y_1 +\cdots + x_dy_d.
$$
The dual $\widehat{K}^d$ consists of all $\chi_y (\cdot)$
with $y \in K^d$, where $\chi_y (x)=\chi(y\cdot x)$.

Let $\mu$ be a finite Borel measure on $K^d$. The {\em Fourier transform} of
$\mu$ is defined to be
$$
    \widehat{\mu} (y)
    = \int_{K^d} \overline{\chi}_y(x) d\mu(x) \qquad (y \in \widehat{K}^d \simeq K^d).
$$
The Fourier transform $\widehat{f}$ of $f \in L^1(K^d)$ is that of $\mu_f$ where $\mu_f$ is the measure
defined by $d\mu_f = f d\mathfrak{m}$.
For  local fields and Fourier analysis on them, we can refer to \cite{Cassels1986,Ramakrishnan-Valenza1999,Taibleson1975,Vladimirov-Volovich-Zelenov1994}.

In this paper, we propose to study spectral measures and their variants on $K^d$.
Let $\mu$ be a probability Borel measure on $K^d$. We say that $\mu$ is a {\em spectral measure}
if there exists a set $\Lambda \subset \widehat{K}^d$ such that
$\{\chi_\lambda\}_{\lambda \in \Lambda}$ is an orthonormal basis (i.e. a Hilbert basis) of $L^2(\mu)$. Then
$\Lambda$ is called a {\em spectrum} of $\mu$ and we call $(\mu, \Lambda)$ a {\em spectral pair}.

Assume that $\Omega$ is a set in $K^d$
of positive and finite Haar measure. When the restricted measure $\frac{1}{\mathfrak{m}(\Omega)}\mathfrak{m}|_\Omega$
is a spectral measure, we say $\Omega$ is a {\em spectral set}. In this case, instead of saying $(\frac{1}{\mathfrak{m}(\Omega)}\mathfrak{m}|_\Omega, \Lambda)$ is a spectral pair,  we say that
$(\Omega, \Lambda)$ a {\em spectral pair}.

The characters $\chi_\lambda$ ($\lambda \in \widehat{K}^d$) are called exponential
functions on $K^d$.
The existence of exponential Hilbert basis like $\{\chi_\lambda\}_{\lambda\in \Lambda}$
of $L^2(\mu)$
is a strong constraint on the measure.
Here are some weaker requirements.
The set
$\{\chi_\lambda\}_{\lambda \in \Lambda}$ is a {\em Fourier frame} of $L^2(\mu)$ if there exist constants
$A>0$ and $B>0$ such that
\begin{equation}\label{FF}
   A \|f\|^2 \le \sum_{\lambda} |\langle f, \chi_\lambda\rangle_\mu|^2 \le B \|f\|^2,
   \qquad \forall f \in L^2(\mu)
\end{equation}
where $\langle \cdot, \cdot\rangle_\mu$ denotes the inner product in $L^2(\mu)$.
We say $\{\chi_\lambda\}_{\lambda \in \Lambda}$ is a {\em Riesz basis} if it is a Fourier frame as well as a Schauder basis. When  $\{\chi_\lambda\}_{\lambda \in \Lambda}$ is a Riesz basis (resp. Fourier frame) of $L^2(\mu)$, $\mu$ is called {\em  $R$-spectral measure} (resp. {\em $F$-spectral measure}) and $\Lambda$ is called a {\em $R$-spectrum} (resp. {\em  $F$-spectrum}).

If $\{\chi_\lambda\}_{\lambda \in \Lambda}$ only satisfies
the first inequality (resp. the second inequality) in (\ref{FF}), we say that $\Lambda$ is a {\em set of sampling} (resp. a {\em Bessel sequence}) of $L^2(\mu)$.
If for any sequence $\{a_\lambda\}_{\lambda\in \Lambda}\in \ell^2(\Lambda)$,
there exists $f\in L^2(\Omega)$ such that $a_\lambda = \widehat{f}(\lambda)$
for all $\lambda\in \Lambda$, we say that $\Lambda$ is a {\em set of interpolation}
for $L^2(\Omega)$.


An obvious necessary condition for $\mu$ to be a spectral measure with
$\Lambda$ as spectrum is
\begin{equation}\label{orthogonality}
(\Lambda -\Lambda)\setminus \{0\} \subset \mathcal{Z}_\mu: =\{\xi \in \widehat{K}^d: \widehat{\mu}(\xi) =0\}.
\end{equation}
This is actually a necessary and sufficient condition for
$\{\chi_\lambda\}_{\lambda\in \Lambda}$ to be orthogonal in $L^2(\mu)$, because
$$
      \langle \chi_\xi, \chi_\lambda\rangle_\mu = \int \chi_\xi \overline{\chi}_\lambda d\mu = \widehat{\mu}(\lambda -\xi).
  $$

Here is a criterion for $\mu$ to be a spectral measure.

\begin{theorem}\label{Thm-SpectralMeasure} A Borel probability measure on $K^d$ is a spectral measure with $\Lambda \subset
\widehat{K}^d$ as its spectrum iff
\begin{equation}\label{spectral criterion}
\forall \xi \in \widehat{K}^d, \quad \sum_{\lambda \in \Lambda}
 |\widehat{\mu} (\lambda -\xi)|^2 = 1.
 \end{equation}
\end{theorem}

This theorem in the case $\mathbb{R}^d$
is due to Jorgensen and Pedersen \cite{Jorgensen-Pedersen1998}.
The following theorem says that a Bessel sequence perturbed by a bounded sequence
remains a Bessel sequence. The corresponding result in $\mathbb{R}^d$  was proved by Dutkay, Han, Sun and Weber in \cite{Dutkay-Han-Sun-Weber2011}. The proof in $\mathbb{R}^d$ seems not adaptable to the case of local fields. Our proof will based on the fact that characters in local fields are constant in a neighborhood of the origin.

\begin{theorem}\label{Thm-Bessel} Let $\{\lambda_n\}$ be a Bessel sequence of $L^2(\mu)$
  where $\mu$ is of compact support. Let $\{\gamma_n\}$ be another sequence.  Suppose there exists a constant
  $C>0$ such that
  $$
     \forall n, \quad |\gamma_n -\lambda_n|\le C.
  $$
  Then $\{\gamma_n\}$ is also a Bessel sequence of $L^2(\mu)$.
  \end{theorem}

The following theorem is a version in local fields of Landau's density theorem
which establishes relationship between
the set of sampling and the set of interpolation $\Lambda$ and the Beurling densities $D^+(\Lambda)$, $D^-(\Lambda)$, $D(\Lambda)$ (see Section 5 for the definition). A set $\Lambda \subset K^d$ is said to be (uniformly) {\em discrete} if
$$
   d(\Lambda): = \inf_{\sigma, \tau \in \Lambda; \sigma\not=\tau} |\sigma-\tau| >0.
$$
Any number $\delta$ with $0<\delta\le d(\Lambda)$ will be called separation constant of
$\Lambda$.

\begin{theorem}\label{Thm-Landau} Let $\Omega \subset K^d$ be a Borel set such that
$0<\mathfrak{m}(\Omega)<\infty$ and let $\Lambda\subset \widehat{K}^d$ be a discrete set.
\\
\indent {\rm (1)} \ If $\Lambda$ is a set of sampling of $L ^2(\Omega)$, then
$D^-(\Lambda)\ge \mathfrak{m}(\Omega)$.\\
\indent {\rm (2)} \ If $\Lambda$ is a set of interpolation of $L ^2(\Omega)$, then
$D^+(\Lambda)\le \mathfrak{m}(\Omega)$.\\
\indent {\rm (3)} \ If $\Lambda$ is a $F$-spectrum of $\Omega$, then
$D(\Lambda)=\mathfrak{m}(\Omega)$.
\end{theorem}

The above two theorems are fundamental. They allow us to establish the following two basic results on spectral measures.

\begin{theorem} A compactly supported $F$-spectral measure on $K^d$ must be of pure type
in the sense that it is either finitely discrete, or singularly continuous or absolutely
continuous.
\end{theorem}

This result of pure type in $\mathbb{R}^d$ is due to He, Lai and Lau \cite{He-Lai-Lau2013}.
The following boundedness of the density of absolutely continuous F-spectral measures
in $\mathbb{R}^d$ is due to Lai \cite{Lai2011}.

\begin{theorem} Let $\mu$ be a compactly supported and absolutely
continuous probability measure on $K^d$ with Radon-Nikodym derivative $\phi$. If $\mu$ is a $F$-spectral measure, then $A\le \phi(x)\le B$ almost everywhere on the support
of $\mu$, where $0<A\le B<\infty$ are two constants.
\end{theorem}

The existence of spectral pair $(\Omega, \Lambda)$ on Euclidean space $\mathbb{R}^d$
goes back to Segal's problem of commutativity  of the partial derivative operators (1958). Fuglede \cite{Fuglede1974} proved that the commutativity is equivalent to  the existence of spectrum. Fuglede also proved that it is the case for the region that is a lattice tile.
Fuglede conjecture states that it is the case iff the region is a tile (not necessarily
a lattice tile). But Tao \cite{Tao2004} disproved this for $d\ge 5$. Jorgensen and Pedersen
\cite{Jorgensen-Pedersen1996,Jorgensen-Pedersen1998} discovered the first singular spectral
measure (a self-similar measure). Their works were followed by Strichartz \cite{Strichartz1998,Strichartz2000}, and {\L}aba and Wang \cite{Laba-Wang2002}, and many others. In the non-Archimedean case, there are many to do.

We will sketch the proofs of Theorems 1.1, 1.2 and 1.3. respectively in Sections 3, 4 and 5.
Theorems 1.4 and 1.5 can be proved as in the  Archimedean case, by using Theorems 1.2 and 1.3, but
details are not given here. Before proving  Theorems 1.1, 1.2 and 1.3,
in Section 2, we will give some preliminaries concerning the structure of local fields, quasi-lattices and some Fourier integrals.
\medskip


{\em Addendum.} There is a progress in the field of $p$-adic numbers \cite{FFS}, where it is proved that {\em a compact open set in $\mathbb{Q}_p$ is a spectral set if and only if
it is a tile}. This tile must be lattice tile and it is characterized by a special homogeneity.   Without loss of generality, we can only consider  compact open sets in $\mathbb{Z}_p$. Let $\Omega = T \oplus p^n \mathbb{Z}_p$ be a compact  open set in $\mathbb{Z}_p$
where $T\subset \mathbb{Z}_p$ be a finite set, which can be assumed to be a subset of
$\mathbb{N}$ or of $\mathbb{Z}/p^n\mathbb{Z}$, and $n\ge 1$ is an integer. The notation $\oplus$ means that $\Omega$ is a disjoint union of $t + p^n \mathbb{Z}_p$ ($t\in T$). The homogeneity of $\Omega$ is the
homogeneity of $T$ described by (c) in the following theorem which gives a characterization of spectral sets in $\mathbb{Z}/p^n\mathbb{Z}$. See Fig. 1 for a geometric representation of
this kind of homogeneity.

\begin{theorem}[\cite{FFS}]
	Let $T\subset \mathbb{Z}/p^n\mathbb{Z}$. The following propositions are equivalent: \\
  \indent {\rm (a)} \ $T$ is a spectral set in $\mathbb{Z}/p^n\mathbb{Z}$;\\
  \indent {\rm (b)} \	$T$ is a tile of $\mathbb{Z}/p^n\mathbb{Z}$;\\
  \indent {\rm (c)} \	 For any $ i=1, 2, \cdots, n-1$,
   ${\rm Card} (T\!\! \mod{p^i})=p^{k_i}$ for some integer $k_i\in \mathbb{N}$.
\end{theorem}

\begin{figure}
  \centering
  \includegraphics[width=1\textwidth]{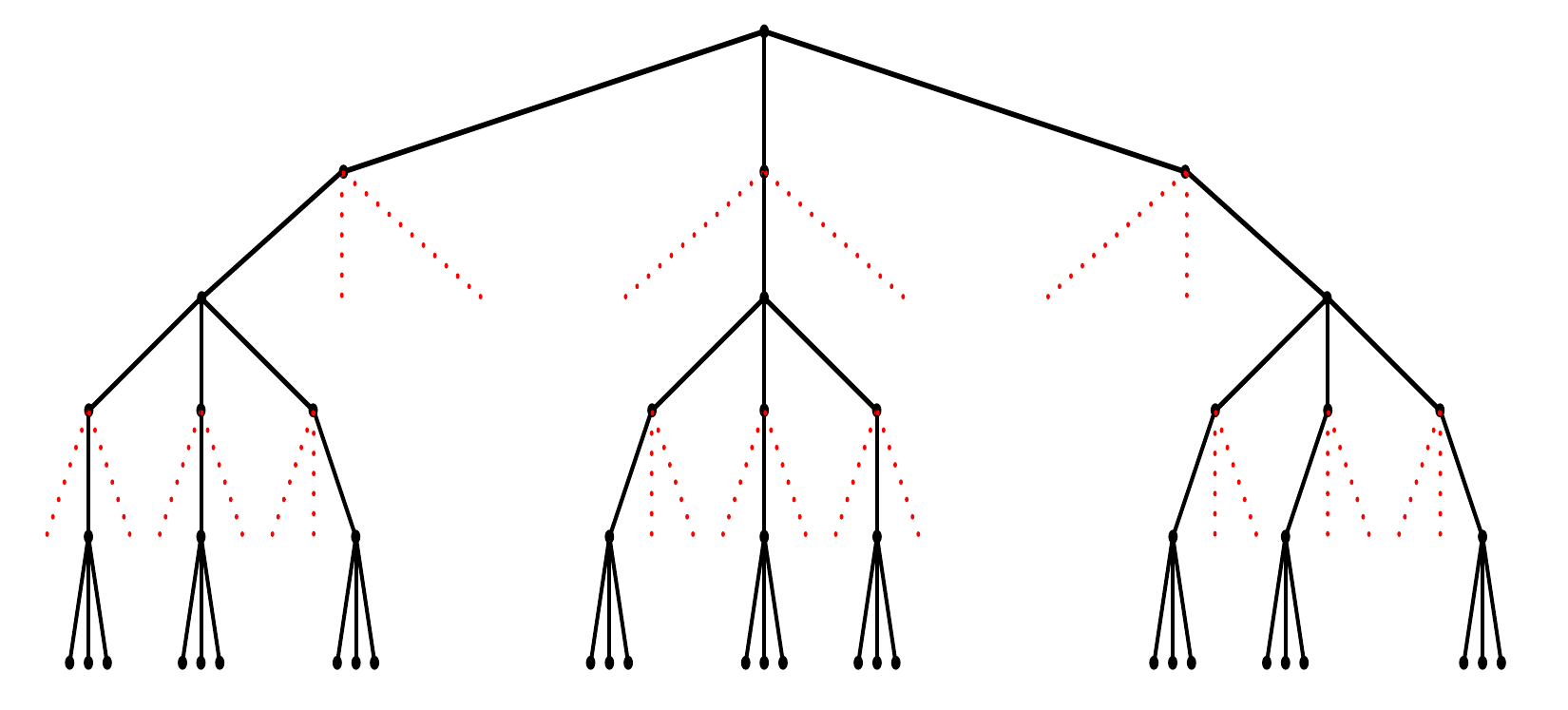}\\
   \centering \caption{A homogeneous set in $\mathbb{Z}/3^5\mathbb{Z}$. From all nodes of a given level, we choose either $3$ branches or $1$ branches. }
\label{TIJ}
\end{figure}

One example of spectral set in $\mathbb{Q}_2$ is
$$
    \Omega = \bigsqcup_{c\in C} (c + 2^3\mathbb{Z}_2)
$$
where $C=\{0, 3, 4, 7\}$. The homogeneous structure of $\{0, 3, 4, 7\}$ can also be seen from
$$
    0= 0\cdot 1 + 0\cdot 2 + 0 \cdot 2^2, \qquad 3= 1\cdot 1 + 1\cdot 2 + 0 \cdot 2^2$$
    $$
    4= 0\cdot 1 + 0\cdot 2 + 1 \cdot 2^2,\qquad 7= 1\cdot 1 + 1\cdot 2 + 1 \cdot 2^2.
$$

Some singular spectral measures are also found in the field $\mathbb{Q}_p$ in \cite{FFS}. Here is an example. Consider the iterated function system defined by
$$
     f_c(x) = 8x + c \quad (c \in C)
$$
where $C$ is the same set $\{0, 3, 4, 7\}$ as above. The invariant measure associated to the
probability $(\frac{1}{4}, \frac{1}{4}, \frac{1}{4}, \frac{1}{4})$ is a spectral measure, which is supported by a Cantor set of dimension $\frac{\log 4}{\log 8}=\frac{2}{3}$.

\medskip
{\em Acknowledgement} My thanks go to Kasing LAU, Shilei FAN, Lingmin LIAO for their
careful reading of the first version of the paper and for their remarks.

\section{Preliminaries}

\subsection{Local fields}
Recall that a local field is a non-discrete locally compact disconnected
field. If it is of characteristic zero, it is a field of p-adic numbers $\mathbb{Q}_p$ or its finite extension. If it is of positive characteristic, it is a field of p-series $\mathbb{F}_p((T))$ or its finite extension $\mathbb{F}_{p^c}((T))$.
Connected locally compact fields are $\mathbb{R}$ and $\mathbb{C}$.

Consider the field $\mathbb{Q}$ of rationals and a prime $p\ge 2$.
Any nonzero number $r\in \mathbb{Q}$ can be written as
$r =p^v \frac{a}{b}$ where $v, a, b\in \mathbb{Z}$ and $(p, a)=1$ and $(p, b)=1$
(here $(x, y)$ denotes the greatest common divisor of two integers $x$ and $y$). By unique factorization in $\mathbb{Z}$,
the number $v$ depends only on $r$. We define $v_p(r)=v$ and
$|r|_p = p^{-v_p(r)}$ for $r\not=0$ and $|0|_p=0$.
Then $|\cdot|_p$ is a non-Archimedean absolute value. That means\\
\indent (i)  \ \ $|r|_p\ge 0$ with equality only for $r=0$; \\
\indent (ii) \ $|r s|_p=|r|_p |s|_p$;\\
\indent (iii) $|r+s|_p\le \max\{ |r|_p, |s|_p\}$.\\
The field of p-adic numbers $\mathbb{Q}_p$ is the completion of $\mathbb{Q}$ under
$|\cdot|_p$. Actually
a typical element of $\mathbb{Q}_p$ is of the form
$$
     \sum_{n= N}^\infty a_n p^{n} \qquad (N\in \mathbb{Z}, a_n \in \{0,1,\cdots, p-1\}).
$$
(the partial sum from $N$ to $m$ is  a fundamenta1 sequence of elements of $\mathbb{Q}$).

Let $\mathbb{F}_q[T]$ be the ring of polynomials over the finite field $\mathbb{F}_q$ of $q=p^c$ elements.
Let $\mathbb{F}_q(T)$ be the field of rational functions of the
indeterminate $T$ with coefficients in $\mathbb{F}_q$. Any nonzero $h\in \mathbb{F}_q(T)$
can be written $
h(T) = T^\rho \frac{f(T)}{g(T)}$
where $\rho\in \mathbb{Z},  f\in \mathbb{F}_q[T]$ and $f\in \mathbb{F}_q[T]$
 with $f(0) \not=0$ and $g(0) \not= 0$.  Then we define
$|h|= q^{-\rho}$. This function $|\cdot|: \mathbb{F}_q(T) \to \mathbb{R}_+$
is also a non-Archimedean absolute value. The completion of $\mathbb{F}_q(T)$, denoted by
$\mathbb{F}_q((T))$,
is called the field of formal Laurent series over $\mathbb{F}_q(T)$. Actually
a typical element of $\mathbb{F}_q((T))$ is of the form
$$
     \sum_{n= N}^\infty a_n T^{n} \qquad (N\in \mathbb{Z}, a_n \in \mathbb{F}_q).
$$
(the partial sum from $N$ to $m$ is  a fundamental sequence of elements of $\mathbb{F}_q(T)$).
\medskip

Let us compares the two fields $\mathbb{Q}_p$
and $\mathbb{F}_q((T))$ in the following table.

\begin{tabular}{|c||c|c||c|c|}
  \hline
  K &$\mathbb{Q}_p$& $\sum_{n= N}^\infty a_n p^{n}$&$\sum_{n= N}^\infty a_n T^{n}$&$\mathbb{F}_q((T))$\\
  \hline
  Completion of & $\mathbb{Q}$ & a/b & f/g & $\mathbb{F}_q(T)$ \\
  \hline
  $\mathfrak{D}$ & $\mathbb{Z}_p$ & $\sum_{n= 0}^\infty a_n p^{n}$ & $\sum_{n= 0}^\infty a_n T^{n}$ & $\mathbb{F}_q[[T]]$ \\
  \hline
  $\mathfrak{P}$ & $(p)$ & $\sum_{n= 1}^\infty a_n p^{n}$ & $\sum_{n= 1}^\infty a_n T^{n}$ & (T) \\
  \hline
  $\mathbf{k}$ & $\mathbb{F}_p$ &  &  & $\mathbb{F}_q$ \\
  \hline
\end{tabular}
\medskip

Formally the two fields $\mathbb{Q}_p$ and $\mathbb{F}_q((T))$
seem the same. But their algebraic operations are different.
In $\mathbb{Q}_p$, we add and multiply coordinate by coordinate but with carry to the right.
In $\mathbb{F}_q((T))$, we add coordinate by coordinate and we multiply by the
rule of Cauchy as we do for polynomials.

In $\mathbb{Q}_p$, a non-trivial character is defined by
$$
    \chi(x) = e^{2\pi i \{x\}}
$$
where $\{x\}= \sum_{n=N}^{-1} a_n p^n$ is the fractional part of $x=\sum_{n=N}^{\infty} a_n p^n$. From this character we can get all characters of $\mathbb{Q}_p$.

In $\mathbb{F}_q((T))$, a non-trivial character is defined by
$$
    \chi(x) = e^{ \frac{2\pi i}{p} v(x)}
$$
where  $x=\sum_{n=N}^{\infty} a_n T^n$ and $v(x)=a_{-1}^{(1)}$
is the first coordinate of $a_{-1}$ in a fixed basis of the $\mathbb{F}_p$-vector
space $\mathbb{F}_q$. This character is trivial on $\mathbb{F}_q[[T]]$ but non-trivial on $T^{-1} \mathbb{F}_q[[T]]$.

\subsection{Notation and Basic facts}
\, \ \ \\
\noindent Notation.
\\ \indent
$\mathfrak{D}^\times := \mathfrak{D}\setminus \mathfrak{P}=\{x\in K: |x|=1\}$.
It is the group of units of $\mathfrak{D}$.

$\mathfrak{U}_n:= 1 + \mathfrak{p}^n \mathfrak{D}$ ($n\ge 1$).
These are subgroups of $\mathfrak{D}^\times$.

$B(0, q^{n}): = \mathfrak{p}^{-n} \mathfrak{D}$.  It is the (closed) ball centered at $0$ of radius $q^n$.

$B(x, q^{n}): = x + B(0, q^{n})$. We also use it to denote balls in $K^d$.

$S(x, q^{n}): = B(x, q^{n}) \setminus B(x, q^{n-1}) $, the sphere  of radius $q^n$.

$\mathcal{A}_n:$ the set of finite union of balls of radius $q^n$ ($n\in \mathbb{Z})$.

$1_A:$ the characteristic function of a set $A$.

\medskip
\noindent Facts.
\\ \indent
$\mathfrak{m}(\mathfrak{P})= q^{-1}$, $|\mathfrak{p}|=q^{-1}$,
$\mathfrak{m}(B(x, q^n)= q^n$.

All $\mathfrak{p}^n \mathfrak{D}$ are additive groups.

All $\mathfrak{U}_n$ are multiplicative groups.

$d \mathfrak{m}(ax) =|a|d\mathfrak{m}(x)$ for all $a\in K^*$. It is the image of $\mathfrak{m}$
under $x\mapsto ax$.


\subsection{Quasi-lattices} Not like in $\mathbb{R}^d$, there is no lattice groups
in $K^d$, because finitely generated additive groups in $K^d$ are bounded.
We define quasi-lattices which will play the role of lattices in $\mathbb{R}^d$.

   The unit ball $B(0, 1)$  is an additive subgroup of $K^d$. Let $\mathbb{L}\subset K^d$
   be a complete set of representatives of the cosets of $B(0, 1)$. Then
   $$
        K^d = \mathbb{L} + B(0, 1) = \bigsqcup_{\gamma\in \mathbb{L}} (\gamma + B(0, 1)).
   $$
   We call $\mathbb{L}$ a {\em standard quasi-lattice} in $K^d$.
   Recall that $\mathbb{Z}^d$ is the standard lattice in $\mathbb{R}^d$, which is
   a finite generated subgroup of $\mathbb{R}^d$. Notice that $\mathbb{L}$ is not a subgroup of $K^d$, as we shall see.

   If $\mathbb{L}$ is a standard quasi-lattice in $K$, then $\mathbb{L}^d$
   is a standard quasi-lattice of $K^d$. If $\mathbb{L}$ is a standard quasi-lattice
   of $K^d$, so is $\{\gamma + \eta_\gamma: \gamma\in \mathbb{L}\}$
   where $\{\eta_\gamma\}_{\gamma\in \mathbb{L}}$ is any set in the unit ball $B(0, 1)$.

   Let us present a standard quasi-lattice in $\mathbb{Q}_p$. For any
   $n\ge 1$, let
   $$
       V_n = \{1\le k <p^n: (k, p) =1\}.
   $$
   The set $V_n$ is nothing but the set of invertible elements of the ring $\mathbb{Z}/p^n \mathbb{Z}$, i.e. $V_n = (\mathbb{Z}/p^n\mathbb{Z})^\times$. Then
   $$
       \{0\} \sqcup  p^{-1}V_1\sqcup p^{-2}V_2 \sqcup\cdots p^{-n}V_n \sqcup \cdots
   $$
   is a standard quasi-lattice of $\mathbb{Q}_p$.

   \begin{lemma} The set of characters of $\mathbb{Z}_p$
   is $\{\chi(\gamma x)\}_{\gamma\in \mathbb{L}}$ where $\mathbb{L}$
   is a standard quasi-lattice of $\mathbb{Q}_p$, where $\chi(x) = e^{2\pi i \{x\}}$.
   \end{lemma}

   As consequence, we get immediately the characters of $\mathbb{Z}_p^d$.
   In another words, $\{\chi(\gamma \cdot x)\}_{\gamma\in \mathbb{L}}$ where $\mathbb{L}$
   is a standard quasi-lattice of $\mathbb{Q}_p^d$ is the set of  characters of $\mathbb{Z}_p^d$. It is a Hilbert basis of $L^2(\mathbb{Z}_p^d)$. It is also
   a Hilbert basis of $L^2(\mathbb{Z}_p^d+ a)$, for any $a\in \mathbb{Q}_p^d$.
   In other words, $(a+\mathbb{Z}_p^d, \mathbb{L})$ is a spectral pair.

   For the group $p^{-n}\mathbb{Z}_p$ ($n\in \mathbb{Z}$), the characters are described
   by $p^{n} \mathbb{L}$. More generally, let $A\in {\rm GL}_d(\mathbb{Q}_p)$ be a non-singular $d\times d$-matrix. The characters of the group
   $A\mathbb{Z}_p^d$ are described by $(A^{-1})^t \mathbb{L}$.
   We call such a set $(A^{-1})^t \mathbb{L}$ a {\em quasi-lattice} of $\mathbb{Q}_p^d$.

   Quasi-lattices are separated.  Standard quasi-lattice $\mathbb{L}$ in
   $K^d$ admits its separation constant $d(\mathbb{L})=q$.
   Let us give a direct proof for the standard quasi-lattice in $\mathbb{Q}_p$.

   \begin{lemma} We  have $d(\mathbb{L})=p$ for the standard quasi-lattice $\mathbb{L}$ in $\mathbb{Q}_p$.
   \end{lemma}

   \begin{proof} Let $
       \mathbb{L}_0 = \{0\}$ and $\mathbb{L}_n =\{p^{-n} k: 1\le k<p^n, (k, p)=1\}
   $ ($n\ge 1$).
   It is clear that $\mathbb{L}_n \subset S(0, p^n)$ for $n\ge 1$. For
   $\lambda'\in \mathbb{L}_n$ and $\lambda''\in \mathbb{L}_m$ ($n<m$), we have
   $|\lambda'-\lambda''| =p^m$. In fact, assume $\lambda'=p^{-n}k_1$ and $\lambda''=p^{-m}k_2$.
   Then
   $$
      \lambda'-\lambda'' = p^{-m} (k_1 p^{m-n} + k_2),
      \qquad    (p, k_1 p^{m-n} + k_2)=1.
   $$
   For
   $\lambda', \lambda''\in \mathbb{L}_n$ with $\lambda'\not=\lambda''$, we have
   $|\lambda'-\lambda''| \ge p$. In fact, assume $\lambda'=p^{-n}k_1$ and $\lambda''=p^{-n}k_2$.
   Then  $\lambda'-\lambda'' = p^{-n}(k_1 -k_2)$. Since $1\le |k_1-k_2|<p^n$,  $k_1-k_2$
   is not divisible by $p^n$. On the other hand, we have
   $|\lambda'-\lambda''|=p$ for $\lambda'= p^{-n}$ and $\lambda''=p^{-n}(p^{n-1} +1)$.
   \end{proof}

\subsection{Some Fourier integrals}

The Fourier transform of the characteristic function of a ball centered at $0$
is a function of the same type.

 \begin{lemma} We have $\widehat{1_{B(0, q^a)}}(\xi)= q^a 1_{B(0, q^{-a})} (\xi)$
 for any $a\in \mathbb{Z}$.
 \end{lemma}
 \begin{proof} By the scaling property of the Haar measure, we have to prove the result
 in the case $a=0$. Recall that
 $$
    \widehat{1_{B(0, 1)}}(\xi)
    = \int_{B(0, 1)} \chi(-\xi \cdot x) dx.
 $$
 When $|\xi|\le 1$, the integrand is equal to $1$, so $\widehat{1_{B(0, 1)}}(\xi)=1$.
 When $|\xi|>1$, making a change of variable $x = y-z$ with $z\in B(0, 1)$ chosen
 such that $ \chi(\xi\cdot z)\not=1$,
 we get $$\widehat{1_{B(0, 1)}}(\xi) = \chi(\xi\cdot z) \widehat{1_{B(0, 1)}}(\xi).$$
 It follows that $\widehat{1_{B(0, 1)}}(\xi)=0$ for $|\xi|>1$.
 \end{proof}

  \begin{lemma} Let $O = \bigsqcup_j B(\tau_j, q^a)\in \mathcal{A}_a$ be a finite union of ball of the same size,
 where $a\in \mathbb{Z}$. We have
 $$\widehat{1_O}(\xi) = q^a 1_{B(0, q^{-a})} (\xi) \sum_j \chi(- \xi\cdot \tau_j).$$
 In particular, $\widehat{1_O}(\xi)$ is supported by the ball $B(0, q^{-a})$.
 \end{lemma}
 \begin{proof} It is a direct consequence of the last lemma.
 \end{proof}

  \begin{lemma}\label{D-int} For $a, b\in \mathbb{Z}$, we have
  $$
      \int_{|\xi|\le q^a} \int_{|\eta|\le q^a}
      |\widehat{1_{B(0, q^b)}}(\xi-\eta)|^2d\xi d\eta =
      \left\{
      \begin{array}{lcl}
        q^{a+b} & \mbox{\rm if} &  a+b \ge 0 \\
        q^{2(a+b)} & \mbox{\rm if} & a+b <0.
      \end{array}
      \right.
  $$
  \end{lemma}

  \begin{proof} Recall that $\widehat{1_{B(0, q^b)}}(\xi)= q^b 1_{B(0, q^{-b})} (\xi)$.
  Using this and making the change of variables $\xi = \mathfrak{p}^b u, \eta = \mathfrak{p}^b v$ (the jacobian is equal to $q^{-2b}$). The integral
  becomes
  $$
        \int_{|u|\le q^{a+b}} \int_{|v|\le q^{a+b}}
      |{1_{B(0, 1)}}(u-v)|^2du dv.
  $$

  Assume $a+b<0$. Then the ball $B(0, q^{a+b})$ is contained in the unit ball $B(0, 1)$
  so that the integrand is equal to $1$ on $B(0, q^{a+b}) \times B(0, q^{a+b})$.
  The integral is then equal to $q^{2(a+b)}$, the Haar measure of $B(0, q^{a+b}) \times B(0, q^{a+b})$.

  Assume now $a+b\ge 0$. The ball $B(0, q^{a+b})$ is the disjoint union of the balls $B(c, 1)$ with center
  in $\mathbb{L}_{a+b}: = \mathbb{L}\cap B(0, q^{a+b})$, where $\mathbb{L}$ is a standard quasi-lattice of $K^d$. So, the above double integral is equal to
  $$
     \sum_{c'\in \mathbb{L}_{a+b}}  \sum_{c''\in \mathbb{L}_{a+b}}
         \int_{B(c',1)} \int_{B(c'', 1)}
      |{1_{B(0, 1)}}(u-v)|^2du dv.
  $$
  If $c'\not= c''$, the balls $B(c',1)$ and $B(c'',1)$ have a distance strictly larger than $1$,
  then the corresponding integral is equal to zero. The integral is equal to $1$
  if $c'=c''$. Thus the above sum is equal to the
  cardinality of $\mathbb{L}_{a+b}$, that is $q^{a+b}$.
  \end{proof}

\section{Criterion of spectral measure}
\subsection{Proof of Theorem \ref{Thm-SpectralMeasure}}
  Recall that $\langle f, g\rangle_\mu$ denotes the inner product in $L^2(\mu)$:
  $$
     \langle f, g\rangle_\mu = \int f \overline{g} d \mu, \quad \forall f, g \in L^2(\mu).
  $$
 Remark that
  $$
      \langle \chi_\xi, \chi_\lambda\rangle_\mu = \int \chi_\xi \overline{\chi}_\lambda d\mu = \widehat{\mu}(\lambda -\xi).
  $$
  It follows that $\chi_{\lambda'}$ and $\chi_{\lambda''}$
  are orthogonal in $L^2(\mu)$ iff $\widehat{\mu}(\lambda -\xi)=0$.

  Assume that $(\mu, \Lambda)$ is a spectral pair. Then  (\ref{spectral criterion}) holds because of the Parseval
  equality and of the fact that $\{\widehat{\mu}(\lambda -\xi)\}_{\lambda\in \Lambda}$
  are Fourier coefficients of $\chi_\xi$ under the Hilbert basis $\{\chi_\lambda\}_{\lambda\in \Lambda}$.

  Now assume (\ref{spectral criterion}) holds. Fix any $\lambda'\in \Lambda$ and
   take $\xi=\lambda'$ in (\ref{spectral criterion}). We get
   $$
       1 + \sum_{\lambda\in \Lambda, \lambda\not=\lambda'}|\widehat{\mu}(\lambda -\lambda')|^2=1,
       $$
   which implies  $\widehat{\mu}(\lambda -\lambda')=0$ for all $\lambda \in \Lambda\setminus \{\lambda'\}$. Thus we have proved the orthogonality of
   $\{\chi_\lambda\}_{\lambda \in \Lambda}$.
   It remains to prove that
  $\{\chi_\lambda\}_{\lambda\in \Lambda}$ is total. By the Hahn-Banach Theorem, what we have to prove is the implication
  $$
      f\in L^2(\mu), \forall \lambda \in  \Lambda, \langle f, \chi_\lambda\rangle_\mu=0
      \Rightarrow f=0.
  $$
  The condition (\ref{spectral criterion}) implies that
  $$
     \forall \xi\in \widehat{K}^d, \quad \chi_\xi = \sum_{\lambda\in \Lambda}
      \langle \chi_\xi, \chi_\lambda\rangle_\mu \chi_\lambda.
  $$
  This implies that $\chi_\xi$ is in the closure of the space spanned by
  $\{\chi_\lambda\}_{\lambda\in \Lambda}$. As $f$ is orthogonal to $\chi_\lambda$
  for all $\lambda \in \Lambda$. So, $f$ is orthogonal to
  $\chi_\xi$. Thus we have proved that
  $$
     \forall \xi \in \widehat{K}^d, \quad \int \overline{\chi}_\xi f d \mu = \langle f, \chi_\xi\rangle_\mu = 0.
  $$
  That is, the Fourier coefficients of the measure $f d\mu$ are all zero. Finally
  $f=0$ $\mu$-almost everywhere.

\subsection{Spectral sets}
 A spectral set $\Omega$ corresponds to a spectral measure $\frac{1}{\mathfrak{m}(\Omega)}\mathfrak{m}|_\Omega$, where we assume
  $0<\mathfrak{m}(\Omega)<\infty$.
  The Fourier transform of this measure is equal to $\frac{1}{\mathfrak{m}(\Omega)} \widehat{1_\Omega}(y)$. Thus, the following is a corollary of Theorem \ref{Thm-SpectralMeasure}.

 \begin{proposition} Suppose $0<\mathfrak{m}(\Omega)<\infty$. Then $\Omega$
   is a spectral set with $\Lambda$ as spectrum iff $$
   \forall \xi \in \widehat{K}^d, \quad \sum_{\lambda \in \Lambda}
 |\widehat{1_\Omega} (\lambda -\xi)|^2 = \mathfrak{m}(\Omega)^2.
   $$
\end{proposition}

\subsection{Finite measures and Hadamard matrices}
A $n\times n$ matrix $H=(h_{i, j})$ is a (complex) {\em Hadamard matrix} if $h_{i,j}\in \mathbb{C}$
with $|h_{i, j}|=1$ such that
$$
    H \ \bar{H}^t = n I
$$
where $I$ is the unit matrix, i.e. $\frac{1}{\sqrt{n}} H$ is unitary.

Let us consider a finite set $S\subset K^d$ of $n$ points and the uniform
probability measure on $S$:
$$
   \mu_S = \frac{1}{n} \sum_{s\in S}\delta_s.
$$
The space $L^2(\mu_S)$ is then of dimension $n$. Let
$\Lambda \subset \widehat{K}^d$ be a set of $n$ points.

\begin{proposition} The pair $(\mu_S, \Lambda)$ is a spectral pair iff $(\chi(\lambda \cdot s))_{\lambda\in \Lambda, s\in S}$ is a Hadamard matrix.
\end{proposition}

It is just because of the fact that for $\lambda', \lambda''\in \Lambda$, we have
$$
   \langle \chi_{\lambda'}, \chi_{\lambda''}\rangle_\mu = \frac{1}{n} \sum_{s\in S}
      \chi(\lambda'\cdot s) \overline{\chi(\lambda''\cdot s)}.
$$

  \section{Perturbation of Bessel sequences}


  In this section, we prove Theorem \ref{Thm-Bessel}. The proof  is based on the following three lemmas.

  The first lemma tells us that we can assume that the measure $\mu$ is supported by
  the unit ball $\mathfrak{D}^d$ withou loss of generality.
  Let $\Omega = \mbox{\rm supp} \mu$. Let $a\in K^*$. Consider
  the map $\tau : \Omega \to a\Omega$ defined by $\tau(x)= ax$. Let
  $\tau_*\mu$ be the image of $\mu$ under $\tau$ (pushed forward)and let
  $\tau^* g$ be the image of a function $g$ under $\tau$ (pulled back). Recall that
  $\tau^* g(x) = g(ax)$ where $g$ is defined on $a\Omega$. Then
  $$
      \int_{a\Omega} g d\tau_*\mu = \int_\Omega \tau^* g d\mu.
  $$
  In particular, $\tau^* : L^2(\tau_*\mu)\to L^2(\mu)$ is an isometry.
  Observe that if $a$ is small enough, we have $a\Omega \subset \mathfrak{D}^d$.
  The first lemma actually establish a relationship between the Bessel sequences of $\mu$
   and those of $\tau_* \mu$. Similar statement holds for spectrum, $F$-spectrum, $R$-spectrum
   etc.

   \begin{lemma}\label{Bessel-scaled} If $\{\lambda_n\}$ is a Bessel sequence of $L^2(\mu)$,
   then $\{a^{-1}\lambda_n\}$ is a Bessel sequence of $L^2(\tau_*\mu)$ with the same constant.
   \end{lemma}

   \begin{proof} The hypothesis means that for some constant $B >0$ we have
   $$
    \forall f \in L^2(\mu),
    \quad \sum |\langle f, \chi(\lambda_n \cdot)\rangle_\mu|^2
    \le B \|f\|^2_\mu
   $$
   Let $g \in L^2(\tau_*\mu)$. We have
\begin{eqnarray*}
   \langle g, \chi(a^{-1}\lambda_n \cdot)\rangle_{\tau_*\mu}
    &=& \int g(y) \overline{\chi(a^{-1}\lambda_n y)} d\tau_*\mu(y)\\
    &=&  \int \tau^*g(x) \overline{\chi(\lambda_n x)} d\mu(x)=
    \langle \tau^*g, \chi(\lambda_n \cdot)\rangle_{\mu}.
\end{eqnarray*}
Thus
$$
 \sum |\langle g, \chi(a^{-1}\lambda_n \cdot)\rangle_{\tau_*\mu}|^2
 =\sum |\langle \tau^*g, \chi(\lambda_n \cdot)\rangle_{\mu}|^2
 \le B \|\tau^* g\|^2 = B\|g\|^2.
$$
   \end{proof}

Thus we can consider $\tau_*\mu$ with a small $a$ such that $a\Omega \subset \mathfrak{D}$.
Then
$$
    |a^{-1}\lambda_n -a^{-1}\gamma_n|\le |a^{-1}| C.
$$

The second lemma is specific for local fields. It says that "very" small perturbation
has no effect.

   \begin{lemma}\label{Bessel-perturbed} Suppose $\mu$ has its support in $\mathfrak{D}^d$.
   If $\lambda_n$ is a Bessel sequence of $L^2(\mu)$, so is
   $\{\lambda_n + \eta_n\}$ for any sequence $\{\eta_n\}$
   such that $|\eta_n|\le 1$.
   \end{lemma}

   \begin{proof}
      It is just because $\chi(\cdot)$ is equal to $1$ on $\mathfrak{D}^d$
      so that
      $\chi((\lambda_n+\eta_n)x) = \chi(\lambda_nx)$ for
      all $x \in \mathfrak{D}^d$.
   \end{proof}

   The following lemma follows directly from the definition.

   \begin{lemma}\label{Decomp}
     Let $\Lambda$ be partitioned into a finite number of subsets
      $\Lambda_1, \cdots, \Lambda_r$. Then $\Lambda$ is a Bessel sequence
      of $L^2(\mu)$ iff every $\Lambda_j$ ($j=1, 2, \cdots, r$) is a Bessel
     sequence of $L^2(\mu)$.
   \end{lemma}

   {\em Proof of Theorem \ref{Thm-Bessel}.} By Lemma \ref{Bessel-scaled},
   we can assume that the measure $\mu$ is supported by $\mathfrak{D}^d$.
   Let $\delta_n = \gamma_n -\lambda_n$. Since
  $|\delta_n|\le C$, we can write
  $$
      \delta_n = t_n + \eta_n \quad \mbox{\rm with}\ t_n \in \Delta, |\eta_n|\le 1
  $$
  where $\Delta$ is a finite set of the standard quasi-lattice. Then
  $$
    \gamma_n = \lambda_n - t_n - \eta_n.
  $$
  For any $k\in \Delta$, let
       $$
          \Lambda_k = \{\lambda_n:   t_n = k\}.
       $$
  By Lemma \ref{Decomp}, $\Lambda_k$ is a Bessel sequence of $L^2(\mu)$
  and so is its translation $\Lambda_k - k$. The sequence $\{\lambda_n -t_n\}$ is nothing but the finite union of  $\Lambda_k - k$ ($k\in \Delta$), which is a Bessel sequence of $L^2(\mu)$, again thanks to Lemma \ref{Decomp}.
   Now we can conclude  by Lemma \ref{Bessel-perturbed}.


\section{Landau operators and Beurling density}
Landau operators on $\mathbb{R}^d$ were introduced and studied by Landau \cite{Landau1967}.
They allowed the establishment of some fundamental relations between
the Beurling density and the sampling property and interpolation property
of discrete set in $\widehat{\mathbb{R}}^d$. We can adapt the theory to the
local fields.

\subsection{Landau operators and basic properties}

Let $\Omega \subset K^d$ and $\Delta\subset \widehat{K}^d$ be two Borel sets.
Assume $$
0<\mathfrak{m}(\Omega)<\infty,\quad
0<\mathfrak{m}(\Delta)<\infty.$$
 Any function $f$ in $L^2(\Omega)$ is identified with the function
which is equal to $f$ on $\Omega$ and to $0$ outside $\Omega$. With this in mind, we have
$L^2(\Omega) \subset L^2(K^d)$. Any function $f$ in $L^2(\Omega)$ is $L^1$-integrable,
its Fourier transform $\widehat{f}$ belongs to $C_0(\widehat{K}^d)$ by the Riemann-Lebesgue Lemma.

A discrete set $\Lambda \subset \widehat{K}^d$ is called a {\em set of sampling}
of $L^2(\Omega)$ if there exists a constant $C>0$ such that
$$
    \forall f \in L^2(\Omega),
    \qquad \|f\|^2 \le C \sum_{\lambda \in \Lambda} |\widehat{f}(\lambda)|^2.
$$
Let
$$
      L^2(\Omega)^{\widehat{}} =\{g \in L^2(\widehat{K}^d): \exists f \in L^2(\Omega) \
      \mbox{\rm such \ that}\ g = \widehat{f} \}.
$$
The inverse Fourier transform of $f\in L^1(K^d)$ is denoted by
$$
   \check{f}(x) =\int f(\xi)\chi_x(\xi) d \xi.
$$

Define $T_\Omega: L^2(\widehat{K}^d) \to L^2(\Omega)^{\widehat{}}$ by
$$
    T_\Omega g = (1_\Omega \check{g})^{\hat{}}.
$$
Let $P_\Delta$ be the orthogonal projection
from $L^2(\widehat{K}^d)$ onto $L^2(\Delta)$, defined by
$$
                P_\Omega g(\xi) = 1_\Delta(\xi) g(\xi).
$$

The {\em Landau operator} $\mathcal{L}=\mathcal{L}_{\Omega, \Delta}:
 L^2(\widehat{K}^d) \to  L^2(\widehat{K}^d)$ is defined by
 $$
        \mathcal{L} = T_\Omega P_\Delta T_\Omega.
 $$
The {\em auxiliary Landau operator} $\mathcal{L}^\sharp: L^2(\widehat{K}^d) \to
L^2(\widehat{K}^d)$ is defined by
$$
    \mathcal{L}^\sharp = P_\Delta T_\Omega P_\Delta.
$$

\begin{lemma}\label{compactness}
      The Landau operator $\mathcal{L}$ is a positive compact self-adjoint operator. It
      is a Hilbert-Schmidt integral operator with kernel
      $$
         K(\eta, \xi) = \int 1_\Delta(t) \widehat{1_\Omega}(\eta -t)
         \overline{\widehat{1_\Omega}(\xi -t)} dt.
      $$
\end{lemma}

\begin{proof}Recall that
    $$
       T_\Omega f(\xi) = \int 1_\Omega(x) \check{f} (x) \overline{\chi(\xi\cdot x)} d x
       =  \int  \check{f} (x) \overline{1_\Omega(x) \chi(\xi\cdot x)} d x.
    $$
    Notice that the Fourier transform of $1_\Omega(x)\chi(\xi\cdot x)$
    is equal to $\widehat{1_\Omega} (t-\xi)$. So, by the Parseval identity, we get
    \begin{equation}\label{T}
       T_\Omega f(\xi) = \int f (t) \widehat{1_\Omega}(\xi -t) d t.
 \end{equation}
    Hence $P_\Delta T_\Omega$ is an integral operator:
  \begin{equation}\label{PT}
       P_\Delta T_\Omega f(\xi) = \int f (t) \Phi(\xi, t) d t
    \end{equation}
    with the kernel $\Phi(\xi, t) = 1_\Delta (\xi)\widehat{1_\Omega}(\xi -t)$, which satisfies
    $$
       \int \int |\Phi(\xi, t)|^2 dt d\xi = \|1_\Delta\|^2\| 1_\Omega\|^2 < \infty.
    $$
    So, $P_\Delta T_\Omega$ is a Hilbert-Schmidt operator hence compact. Thus
    the Landau operator $T_\Omega P_\Delta T_\Omega$ is also compact. Both $T_\Omega$
    and $P_\Delta$ being orthogonal projections, it is then easy to check that
    $\mathcal{L}$ is self-adjoint and positive. From (\ref{T}) and (\ref{PT}), we get
    $$
       \mathcal{L}f(\xi) = \int \left(\int f(\eta) \Phi(t, \eta)  d\eta\right)\widehat{1_\Omega}(\xi -t) dt
       = \int f(\eta) K(\xi, \eta) d \eta
    $$
    where $$K(\xi, \eta)=
    \int \Phi(t, \eta) \widehat{1_\Omega}(\xi -t) dt
    = \int  1_\Delta(t) \widehat{1_\Omega}(t-\eta)\widehat{1_\Omega}(\xi -t) dt.
    $$
\end{proof}

By the above lemma, the Landau operator $\mathcal{L}$ admits a sequence of eigenvalues
that we arrange in decreasing order
$$
   1\ge \lambda_1 \ge \lambda_2 \ge \cdots \ge \lambda_k \ge \cdots \ge 0
$$
with $0$ as the only cluster point, where the multiplicity is taken into account.
The Courant-Fischer-Weyl min-max principle states that the $k$-th eigenvalue is equal to
$$
   \lambda_k = \max_{S_k} \min_{x \in S_k} \frac{\|\mathcal{L} x\|}{\|x\|}
   = \min_{S_{k-1}} \max_{x \in S_{k-1}^\perp} \frac{\|\mathcal{L} x\|}{\|x\|}
$$
where $S_k$ represents an arbitrary subspace of dimension $k$ and $S_{k-1}^\perp$
represents the orthogonal complementary space of $S_{k-1}$.

Let $\{\phi_k\}_{k\ge 1}$ be an orthonormal basis with $\phi_k$ as the eigenvector of
$\lambda_k$. The kernel $K$  is equal to
$$
    K(\eta, \xi) = \sum_{k=1}^\infty \lambda_k \phi_k(\eta) \overline{\phi_k(\xi)}.
$$
It follows that
\begin{eqnarray}
   \sum_{k=1}^\infty \lambda_k
   & = & \int K(\eta, \eta) d \eta,\\
   \sum_{k=1}^\infty \lambda^2_k
   & = & \int\int |K(\eta, \xi)|^2 d \eta d\xi.
\end{eqnarray}

Since the eigenvalue $\lambda_k$ depends on $\Omega$ and $\Delta$, we will denote
it by $\lambda_k(\Omega, \Delta)$. We will study its dependence on $\Omega$ and $\Delta$.

First, we observe that the auxiliary Landau operator $\mathcal{L}^\sharp$
have the same spectrum as the Landau operator  $\mathcal{L}$.

\begin{lemma} $\mathcal{L}^\sharp$ and $\mathcal{L}$ have the same eigenvalues.
\end{lemma}

\begin{lemma} \label{L_property} The eigenvalues
$\lambda_k(\Omega, \Delta)$ ($k =1, 2, \cdots$) of the Landau operator
$\mathcal{L}$ have the following properties:\\
 {\rm (a)}   {\rm (Translation invariance)} $\lambda_k(\Omega, \Delta)
=\lambda_k(\Omega+\sigma, \Delta+\tau)$ for all $\sigma, \tau \in K^d$.\\
 {\rm (b)}   {\rm (Scaling invariance)} $\lambda_k(\Omega, \Delta)
=\lambda_k(a\Omega, a^{-1}\Delta)$ for any $ a \in K^*$.\\
 {\rm (c)}   {\rm (Symmetry)} $\lambda_k(\Omega, \Delta)
=\lambda_k(\Delta, \Omega)$.\\
 {\rm (d)}  {\rm (Monotonicity)} $\lambda_k(\Omega, \Delta_1)
\le \lambda_k(\Omega, \Delta_2)$ if $\Delta_1\subset \Delta_2$.\\
 {\rm (e)}   {\rm (Sum of eigenvalues)} $\sum_{k=1}^\infty
\lambda_k(\Omega, \Delta) = \mathfrak{m}(\Omega) \mathfrak{m}(\Delta)$.\\
{\rm (f)}   {\rm (Sum of eigenvalues squared)} $\sum_{k=1}^\infty
\lambda_k(\Omega, \Delta)^2 = \int_\Delta \int_\Delta |\widehat{1_\Omega}(u-v)|^2 d u dv$.\\
{\rm (g)}   {\rm (Superadditivity)} For $\Delta = \Delta_1 \sqcup\Delta_2$, we have
$$
    \sum_{k=1}^\infty
\lambda_k(\Omega, \Delta)^2 \ge \sum_{k=1}^\infty
\lambda_k(\Omega, \Delta_1)^2 + \sum_{k=1}^\infty
\lambda_k(\Omega, \Delta_2)^2.
$$
{\rm (h)}   {\rm (Weyl-Courant estimates)}
\begin{eqnarray*}
    \lambda_{k+1}(\Omega, \Delta)
    &\le & \sup\left\{\|P_\Delta f\|^2: \|f\|=1, f \in L^2(\Omega)^{\widehat{}}\cap C_k^\perp \right\},\\
    \lambda_{k}(\Omega, \Delta)
    & \ge & \inf\left\{\ \|P_\Delta f\|^2: \|f\|=1, f \in  C_k \right\}
\end{eqnarray*}
where $C_k$ is any $k$-dimensional subspace of $L^2(K^d)$.
\end{lemma}

 \medskip

\subsection{Eigenvalues $\lambda_k(\Omega, \Delta)$ when $\Delta \in \mathcal{A}_a, \Omega \in \mathcal{A}_b$}

We first look at the case where both $\Omega$ and $\Delta$
are balls.

 \begin{lemma}\label{Eigen-ball} Take $\Omega=B(0, q^b)$ and $\Delta = B(0, q^a)$ with $a+b \ge 0$.
 Then the Landau operator $\mathcal{L}$ admits $1$ and $0$ as its only eigenvalues.  The eigenvalue $1$ has multiplicity $q^{a+b}$ and  the eigenvectors associated to $1$ are
 $1_B$'s where $B$ varies among all balls contained in $B(0, q^a)$ of radius $q^{-b}$.
 \end{lemma}

 \begin{proof} By Lemma~\ref{L_property} (e), (f) and Lemma \ref{D-int},
we get
$$
    \sum_{k=1}^\infty \lambda_k = q^{a+b} = \sum_{k=1}^\infty \lambda_k^2.
$$
By the monotonicity of $\lambda_k$ and the fact $0\le\lambda_k\le 1$, the above equality
implies that the first $q^{a+b}$ eigenvalues $\lambda_k$ are equal to $1$ and others are equal to zero.
Assume $c \in B(0, q^a)$. We check that $1_{B(c, q^{-b})} (x)$ is an eigenvector associated to $1$. In fact, since
$B(c, q^{-b}) \subset B(0, q^a)$, we have
\begin{eqnarray*}
    \mathcal{L} 1_{B(c, q^{-b})} (x)
    &= & \int 1_{B(0, q^a)} (x) \widehat{1_{B(0, q^b)}}(x-y) 1_{B(0, q^a)} (y) 1_{B(c, q^{-b})} (y) d y\\
     &= & 1_{B(0, q^a)} (x)  \int q^b 1_{B(0, q^{-b})}(x-y) 1_{B(c, q^{-b})} (y) d y\\
      &= & 1_{B(c, q^{-b})} (x)  \int q^b  1_{B(c, q^{-b})} (y) d y = 1_{B(c, q^{-b})} (x).
\end{eqnarray*}
 \end{proof}

 It follows that  $\mathcal{L}$ is the orthogonal
 projection which can be defined by
 $$
    \mathcal{L} f(x) = \sum_{c \in \mathfrak{p}^{b}\mathbb{L}_{a+b}} 1_{B(c, q^b)}(x) \cdot q^{-db} \int_{B(c, q^b)} f(y) dy.
 $$

\begin{lemma} Let $\Delta \in \mathcal{A}_a$ and $\Omega \in \mathcal{A}_b$ with $a+b\ge 0$. Then
$$
    \lambda_k(\Omega, \Delta) = \left\{
                  \begin{array}{ccl}
                    1 & \mbox{\rm if} &  1\le k \le \mathfrak{m}(\Omega)\mathfrak{m}(\Delta), \\
                    0 & \mbox{\rm if} &   \ \ \ \ \ \, k > \mathfrak{m}(\Omega)\mathfrak{m}(\Delta).
                  \end{array}
    \right.
$$
\end{lemma}

\begin{proof}
     Assume that
     $$
        \Delta = \bigsqcup_{i\in I} B(x_i, q^a), \qquad  \Omega = \bigsqcup_{j\in J} B(y_j, q^b).
     $$
     First we suppose that $\sharp I =1$ and $x_1=0$. Then
     \begin{eqnarray*}
        \mathfrak{m}(\Omega) \mathfrak{m}(B(0, q^a))
        &=& \sum_k \lambda_k(B(0, q^a), \Omega)\\
        &\ge& \sum_k \lambda_k^2(B(0, q^a), \Omega)\\
        &\ge& \sum_j \sum_k\lambda_k^2(B(0, q^a), B(y_j, q^b)).
     \end{eqnarray*}
     By Lemma \ref{Eigen-ball} and the translation invariance, the last sum over $k$
     is equal to $\mathfrak{m}(B(0, q^a))\mathfrak{m}(B(y_j, q^b))$. Thus the last double sum is equal to
     $\sum_j \mathfrak{m}(B(0, q^a))\mathfrak{m}(B(y_j, q^b))$, that is $\mathfrak{m}(\Omega)\mathfrak{m(}B(0, q^a))$.
     Thus we have proved the equality
     $$\sum_k \lambda_k(B(0, q^a), \Omega)
        = \sum_k \lambda_k^2(B(0, q^a), \Omega),$$
        which implies that $\lambda_k(B(0, q^a), \Omega)$ is either $1$ or $0$.

        Then we consider $\lambda_k(\Delta, \Omega)$ by decomposing $\Delta$
        into $B(x_i, q^a)$'s. Using the same argument as above and what we have just proved, we can
        prove the desired result.
\end{proof}

\subsection{Beurling density}

The {\em  upper and lower Beurling density} of a discrete set $\Gamma$ in $K^d$ are defined as follows
\begin{eqnarray*}
    D^+(\Gamma) &=& \limsup_{n\to \infty} \sup_{x\in K^d}\frac{\mathfrak{n}(\Gamma \cap B(x, q^n))}{q^{dn}},
   \\
    D^-(\Gamma) &=& \liminf_{n\to \infty} \inf_{x\in K^d}\frac{\mathfrak{n}(\gamma \cap B(x, q^n))}{q^{dn}}.
\end{eqnarray*}
where $\mathfrak{n}(B)=\sharp B$ is the counting of points in $B$.
As in the Archimedean case \cite{Landau1967}, the above limits exist and we can replace $B(x, q^n)$
by $x + \mathfrak{p}^{-n} I$ where $I$ is any compact set such that $\mathfrak{m}(I)=1$
(the values $D^+(\Gamma)$ and $D^-(\Gamma)$ don't depend on $I$).
If $D^+(\Gamma)= D^-(\Gamma)$, the common value is defined to be the
{\em  Beurling density}.
It is obvious that the Beurling density of a standard quasi-lattice is equal to $1$.

In the following, we give a proof of Theorem \ref{Thm-Landau}. The first lemma is obvious.

\begin{lemma} Let $h\in L^2(\widehat{K}^d)$ be a function such that
$$
\mbox{\rm supp} h \subset B(0, \delta); \qquad \forall x\in \Omega, |\check{h}(x)|\ge 1.$$
Assume $f \in L^2(\Omega)^{\widehat{}}$. Consider
 $$g(y):= f*h(y)= \int_{B(y, \delta)} f(z) h(y-z)dz.
 $$
 Then $ g \in L^2(\Omega)^{\widehat{}}$, $\|g\|\ge \|f\|$ and
$$
    \forall y, \qquad |g(y)|^2 \le \|h\|^2 \int_{B(y, \delta)} |f(z)|^2 dz.
$$
\end{lemma}

\begin{lemma} Let $\Omega \subset K^d$ and $\Delta \subset \widehat{K}^d$ be  bounded Borel sets and let $\Lambda\subset \widehat{K}^d$ be a discrete set with separation constant $\delta >0$. Let $$
\Delta^+ =\{x\in \widehat{K}^d: d(x, \Delta) <\delta/2\},
\qquad \Delta^-= \{x \in \widehat{K}^d: d(x, \Delta^c)>\delta/2\}.$$
\indent {\rm (a)} If $\Lambda$ is a set of sampling of $L^2(\Omega)$, there exists a constant
$0<\alpha<1$ independent of $\Delta$ such that
$$\lambda_{\mathfrak{n}(\Delta^+ \cap \Lambda)+1}(\Omega, \Delta) \le \alpha.
$$
\indent {\rm (b)} If $\Lambda$ is a set of interpolation of $L^2(\Omega)$,
there exists a constant
$\beta>0$ independent of $\Delta$ such that
$$\lambda_{\mathfrak{n}(\Delta^- \cap \Lambda)}(\Omega, \Delta) \ge \beta.
$$
\end{lemma}

\begin{proof} 
Take a function $h$
satisfying the condition required in the last lemma. Such functions $h$ do exist and
we can take $h(\xi) = q^{d \ell} 1_{B(0, q^{-\ell})}(\xi)$ with $\ell$ sufficiently large
such that $q^{-\ell}< \delta/2$ and $\Omega \subset B(0, q^\ell)$. Recall that
$\check{h}(x) = 1_{B(0, q^\ell)}(x)$.

Consider the functions $h(\lambda-\cdot)$ where $\lambda \in  \Delta^+ \cap\Lambda$.
Since $\Lambda$ is separated by $\delta$ and $h$ has support in $B(0, \delta/2)$,
these functions have disjoint supports, then are linearly independent. They span a subspace $C$
of $L^2(\widehat{K}^d)$ of dimension $\mathfrak{n}(\Delta^+ \cap \Lambda)$.

Let $f \in L^2(\Omega)^{\widehat{}} \cap C^\perp$. For $\lambda \in \Delta^+\cap\Lambda$,
since $f$ is orthogonal to $h(\lambda -\cdot)$, we have $g(\lambda)=0$
by the definition of $g$ (see the last lemma). By the last lemma and the sampling property of
$\Lambda$, we get
$$
    \|f\|^2 \le \|g\|^2 \le C \sum_{\lambda\in \Lambda} |g(\lambda)|^2
    = C \sum_{\lambda \in \Lambda \cap \Delta^{+c}} |g(\lambda)|^2.
$$
By the last lemma,
$$
   |g(\lambda)|^2 \le \|h\|^2 \int_{B(\lambda, q^{-\ell})} |f(z)|^2 dz.
$$
Notice that $B(\lambda, q^{-\ell}) \subset \Delta^c$ whence $\lambda \in (\Delta^{+})^c$. So,
$$
  \|f\|^2  \le C\|h\|^2 \int_{\Delta^c} |f(z)|^2 d z = C\|h\|^2 (\|f\|^2 - \|P_\Delta f\|^2).
$$
It follows that $\|P_\Delta f\|^2/\|f\|^2 \le 1-1/(C\|h\|^2)<1$. By the Weyl-Courant estimate,
we get $\lambda_{\mathfrak{n}(\Delta^+\cap \Lambda) +1} \le \alpha$ with
$$
   \alpha = 1-1/(C\|h\|^2)<1.
$$
Thus we have proved the assertion (a). The assertion (b) can be proved
similarly, following a similar argument (see \cite{Landau1967}).
\end{proof}

\begin{lemma} Let  $\Omega \in \mathcal{A}_b$ for some  $b\in \mathbb{Z}$.
Let $\Lambda\subset \widehat{K}^d$ be a discrete set of separation $\delta$. \\
\indent {\rm (a)} Suppose that $\Lambda$ is a set of sampling of $L^2(\Omega)$. Then
there is a constant $c_1>0$ depending only on $\delta$ and $d$ such that for
all $a\in \mathbb{Z}$ with $a+b\ge 0$ we have
$$\min_{x\in K^d}\mathfrak{n}(B(x, q^a) \cap \Lambda)\ge m(\Omega) m(B(0, q^a)
- c_1q^{-a}m(B(0, q^a).$$
\indent {\rm (b)} Suppose that $\Lambda$ is a set of interpolation of $L^2(\Omega)$. Then
there is a constant $c_2>0$ depending only on $\delta$ and $d$ such that for
all $a\in \mathbb{Z}$ with $a+b\ge 0$ we have
$$\max_{x\in K^d}\mathfrak{n}(B(x, q^a) \cap \Lambda)\le m(\Omega) m(B(0, q^a)
+ c_2q^{-a}m(B(0, q^a).$$
\end{lemma}

\begin{proof} Observe that  $\mathfrak{n}(\Delta^-\cap\Lambda)\le \mathfrak{n}(\Delta\cap\Lambda)  \le \mathfrak{n}(\Delta^+\cap\Lambda)$.

(a) Assume that the minimum is attained at $x_0$ and let $\Delta =B(x_0, q^a)$.
Since $a+b\ge 0$,  $1$ is the unique non-zero
eigenvalue of the Landau operator $\mathcal{L}_{\Omega, \Delta}$, which has the multiplicity
$m(G)m(\Delta)$. Then   $$
\mathfrak{n}(\Delta^+\cap\Lambda) +1\ge  m(G)m(\Delta) +1,$$
i.e. $\mathfrak{n}(\Delta^+\cap\Lambda)\ge m(G)m(\Delta)$. However, on one hand
$$
   \mathfrak{n}(\Delta^+\cap\Lambda) = \mathfrak{n}(\Delta\cap\Lambda)+\mathfrak{n}((\Delta^+\setminus\Delta)\cap\Lambda).
$$
On the other hand,
$$
   \mathfrak{n}((\Delta^+\setminus\Delta)\cap\Lambda)
   (\delta/2)^d \le \sum_{\lambda \in(\Delta^+\setminus\Delta)\cap\Lambda}
    \mathfrak{m}(B(\lambda, \delta/2))
$$
which implies that for some constant $c>0$ depending on $\delta$ and $d$ we have
$$
      \mathfrak{n}((\Delta^+\setminus\Delta)\cap\Lambda) \le c q^{-a} \mathfrak{m}(\Delta).
$$
Thus
$$
    \mathfrak{n}(B(x_0, q^a)\cap\Lambda)
    \ge \mathfrak{m}(\Omega)\mathfrak{m}(B(x_0, q^a)) - c q^{-a} \mathfrak{m}(B(x_0, q^a)).
$$

(b) The proof is similar to that of (a), by using (b) in the last lemma.
\end{proof}

Based on the above lemmas, we can prove Theorem \ref{Thm-Landau} by mimicking \cite{Landau1967}.

\setcounter{equation}{0}

\end{document}